\documentclass[11pt,reqno]{amsart}
\usepackage[T1]{fontenc}
\usepackage[utf8]{inputenc}
\usepackage{amssymb,amsmath,amsthm}
\usepackage{enumitem}
\usepackage{tikz}
\usetikzlibrary{decorations.pathreplacing}
\usepackage[margin=1.2in]{geometry}
\usepackage[colorlinks=true,linkcolor=blue,citecolor=blue,urlcolor=blue]{hyperref}

\newtheorem{theorem}{Theorem}[section]
\newtheorem{proposition}[theorem]{Proposition}
\newtheorem{lemma}[theorem]{Lemma}
\newtheorem{corollary}[theorem]{Corollary}
\theoremstyle{definition}

\newtheorem{question}[theorem]{Question}
\theoremstyle{remark}
\newtheorem{remark}[theorem]{Remark}
\usepackage{orcidlink}   % in the preamble

\newcommand{\N}{\mathbb{N}}
\newcommand{\Z}{\mathbb{Z}}
\newcommand{\D}{\mathbb{D}}
\newcommand{\cJ}{\mathcal{J}}
\newcommand{\St}{\widetilde{S}}
\newcommand{\Homeo}{\operatorname{Homeo}}
\newcommand{\id}{\mathrm{id}}
\newcommand{\cs}{2^{\N}}
\newcommand{\xs}{x^{*}}
\newcommand{\zinf}{0^{\infty}}
\newcommand{\oinf}{1^{\infty}}

\begin{document}

\title{On the homeomorphism group of the Cantor set: monotone and measure preserving maps}
%\title{Decomposing the homeomorphism group of the Cantor set into monotone and measure preserving maps}

%\author{Benjamin Vejnar}
\author{Benjamin Vejnar\orcidlink{0000-0002-2833-5385}}
\address{Charles University, Prague, Czech Republic}


\begin{abstract}
Let $E$ be the Cantor set, $G=\Homeo(E)$, let $P\le G$ be the group of order
preserving homeomorphisms and $L\le G$ the group of homeomorphisms preserving the
Cantor--Lebesgue measure $\lambda$. We study how $G$ is built from these two subgroups.
In 2003, Glasner and Weiss stated that $G=LP=PL$. We show that this is not the case, in fact $PL$ and $LP$ are two different proper subsets of $G$.
On the other hand, $PL$ is a dense $G_\delta$ subset of $G$, the
multiplication map $P\times L\to PL$ is a homeomorphism, and $G=PLP=LPL$. In particular, every homeomorphism of $E$ is conjugate to an element of $PL$. We classify the double cosets $PgL$ using special
clopen values set. There are continuum many of them, and $PL$ is the only one which is not meager.
We prove that $G$ is topologically generated by one element of $P$ and one element of $L$ and these two generators produce the Thompson's group $V$. We also show that the generic homeomorphism of $E$ is the product of a generic element of $P$ and a generic element of $L$.
The equality $G=PL$ fails also if $\lambda$ is replaced by any other full nonatomic measure.
\end{abstract}
\renewcommand{\shorttitle}{On the homeomorphism group of the Cantor set}
\maketitle

\section{Introduction}

Let $E\subset[0,1]$ be the Cantor ternary set and let $G=\Homeo(E)$ be the Polish group of
all homeomorphisms of $E$ with the topology of uniform convergence. This group is one of the most
thoroughly studied ``large'' Polish groups, and it has a number of remarkable properties.
As an abstract group, $G$ is simple \cite{An}, all its automorphisms are inner by Rubin's reconstruction theorem \cite{Ru}, and
$G$ is not the union of a countable chain of proper subgroups \cite{DG1}.
%In fact, $G$ has \emph{uncountable strong cofinality} \cite{DG2}: whenever $G$ is the union of an increasing sequence $W_0\subseteq W_1\subseteq\cdots$ of subsets, there are $n$ and $k$ such that $W_n^k=G$, that is, every element of $G$ is a product of $k$ elements of $W_n$.

As a topological group, $G$ is isomorphic to the automorphism group of the countable atomless Boolean algebra $\mathcal B$ of all clopen subsets of $E$, with the topology of pointwise convergence on $\mathcal B$.
Hence $G$ is \emph{non-archimedean} (the open subgroups form a base of neighbourhoods of the identity), and $G$ is \emph{oligomorphic} (for every $n\ge1$ the action of $G$ on $\mathcal B^n$ has only finitely many orbits).
It follows that $G$ is \emph{Roelcke precompact} \cite{Ts} (for every neighbourhood $U$ of the identity there is a finite set $F\subseteq G$ with $G=UFU$), and that $G$ has \emph{Kazhdan's property} (T) \cite{ET}:
there are a compact set $Q\subseteq G$ and $\varepsilon>0$ such that every unitary representation $\pi$ of $G$ which has a unit vector $\xi$ with $\|\pi(g)\xi-\xi\|<\varepsilon$ for all $g\in Q$ has a nonzero invariant vector.

The group $G$ has a dense conjugacy class \cite{GW01}, and even a comeager one \cite{KR}; an explicit description of the generic homeomorphism was given in \cite{AGW} and later in \cite{BernardesDarji, KupkaOprocha}. Moreover, $G$ is topologically $2$-generated \cite[Theorem 2.10]{KR} and it has \emph{ample generics} \cite{Kw} (for every $n\ge1$ the action of $G$ on $G^n$ by simultaneous conjugation of all coordinates has a comeager orbit). Consequently \cite{KR}, $G$ has the \emph{small index property}, (every subgroup of $G$ of index less than $2^{\aleph_0}$ is open),
and every homomorphism of $G$ into a separable topological group is continuous \cite{RS}.

From the dynamical point of view, $G$ is not amenable, since there is no $G$-invariant Borel probability measure on $E$.
In their influential paper \cite{GW} Glasner and Weiss proved that the universal minimal
flow of $G$ is the natural action of $G$ on Uspenskij's space $\Phi(E)$ of maximal chains
of closed subsets of $E$ (see also \cite{KPT, GG}). Section~2 of \cite{GW} contains several results on the structure
of the group $G$ which, as the authors point out, are not needed for the proof of the main result and
are included for completeness. The first of them concerns a decomposition of $G$ into two subgroups which are of independent interest.

Let $P\le G$ be the group of all order preserving homeomorphisms of $E$, and let $L\le G$ be the group of all homeomorphisms of $E$ which preserve the Cantor--Lebesgue measure $\lambda$ on $E$.
The group $P$ is isomorphic to the automorphism group of the ordered set of rational numbers \cite[Theorem 1.2]{GW} and thus it is extremely amenable \cite{Pe} and it has a comeager conjugacy class \cite{Tr}.
The group $L$ has ample generics \cite[Theorem 1.12]{KR}.
Theorem 2.1 of \cite{GW} asserts that $L$ and $P$ are closed subgroups of $G$, that $L\cap P=\{\id\}$, and that
\[
G=LP=PL.
\]
The statements about closedness and about the intersection are correct. However, the
proof of the equality $G=LP=PL$ given in \cite{GW} contains a gap, and the equality
itself turns out to be false. The aim of this paper is to describe what is true instead and, more generally, to study how the group $G$ is built from its subgroups $P$ and $L$.

\begin{theorem}\label{thm:main}
Let $G=\Homeo(E)$ and let $P$ and $L$ be as above.
\begin{enumerate}[label=\textup{(\arabic*)}]
\item\label{main1} $P$ and $L$ are closed subgroups of $G$ and $P\cap L=\{\id\}$. Consequently, every
$g\in PL$ can be written as $g=pl$ with $p\in P$, $l\in L$ in a unique way.
\item\label{main2} $PL\neq LP$. In particular, neither $PL$ nor $LP$ is a subgroup of $G$, and
$PL\neq G\neq LP$.
\item\label{main3} $PL$ and $LP$ are dense $G_\delta$ subsets of $G$, and the multiplication
maps $P\times L\to PL$ and $L\times P\to LP$ are homeomorphisms.
\item\label{main4} $G=PLP=LPL$.
\end{enumerate}
\end{theorem}

Thus a generic homeomorphism of the Cantor set is, in a unique way, a composition of a measure
preserving homeomorphism and an order preserving one (in either order), but not every
homeomorphism is. Since $G=PLP=LPL$, every homeomorphism of $E$ is conjugate to an element of $PL$ (Corollary \ref{cor:conj}). 
%The situation is reminiscent of the Gauss decomposition in $\mathrm{GL}_n$ \cite[Section 3.5]{HJ}, where the products of a lower and an upper triangular matrix fill only a dense open set; the double cosets $PgL$ considered below play the role of the cells of the Bruhat decomposition \cite[Section 28]{Hu}.

The membership in $PL$ is governed by an invariant of measures on ordered Cantor spaces introduced by Akin \cite{Ak1}. A \emph{gap} of $E$ is a pair $(s,t)$ of points of $E$ such that $s<t$ and no point of $E$ lies strictly between $s$ and $t$. Let $\cJ$ be the countable family of all nonempty proper clopen
initial segments of $E$, that is, the sets $E\cap[0,s]$, where $s$ is the left endpoint of a gap.
For a Borel probability measure $\mu$ on $E$, let
\[
\St(\mu)=\{\mu(J):J\in\cJ\}
\]
be its \emph{special clopen values set}. Let $\D$ denote the set of dyadic rationals in $(0,1)$,
so that $\St(\lambda)=\D$. For $g\in G$ we write $g_*\lambda$ for the image measure, $g_*\lambda(B)=\lambda(g^{-1}B)$.
We always have $\St(g_*\lambda)\subseteq\D$, and (Corollary \ref{cor:criterion})
\[
g\in PL\quad\Longleftrightarrow\quad \St(g_*\lambda)=\D .
\]
The gap in \cite{GW} consists in taking the inclusion $\D\subseteq\St(g_*\lambda)$ for granted.
In Section \ref{sec:example} we describe a simple $g\in G$ with $\St(g_*\lambda)=\D\setminus\{\frac12\}$.

Since $\St((pgl)_*\lambda)=\St(g_*\lambda)$ for $p\in P$ and $l\in L$, the set $\St(g_*\lambda)$ depends only on the double coset $PgL$.
Our second result describes the space of double cosets. For $S\subseteq(0,1)$ and $a<b$ in $S\cup\{0,1\}$ let
$\Sigma_S(a,b)$ be the set of all sums $\sum_{i=1}^k(b_i-a_i)$, where $k\ge1$ and
$a\le a_1<b_1\le a_2<b_2\le\dots\le a_k<b_k\le b$ are elements of $S\cup\{0,1\}$. We say that
$S$ is \emph{admissible} if $S\subseteq\D$, $S$ is dense in $(0,1)$, and
\begin{equation}\label{eq:adm}
\D\cap(0,b-a)\subseteq\Sigma_S(a,b)\qquad\text{for all $a<b$ in $S\cup\{0,1\}$.}
\end{equation}

\begin{theorem}\label{thm:cosets}
The assignment $PgL\mapsto\St(g_*\lambda)$ is a bijection between the set $\{PgL:g\in G\}$ of
double cosets and the family of all admissible subsets of $\D$. The double coset $PL$ corresponds to
$\D$. If $R\subseteq\D$ has no accumulation point in $(0,1)$, then $\D\setminus R$ is admissible.
Consequently, there are $2^{\aleph_0}$ double cosets, and all of them except $PL$ are meager in $G$.
\end{theorem}

As mentioned above, $G$ is topologically $2$-generated \cite[Theorem 2.10]{KR}. Our third result shows that the two generators can be taken from $P$ and $L$.

\begin{theorem}\label{thm:gen}
Let $p\in P$ and $l\in L$ be defined by
\begin{gather*}
p(0y)=00y,\qquad p(10y)=01y,\qquad p(11y)=1y,\\
l(00y)=01y,\qquad l(01y)=11y,\qquad l(11y)=10y,\qquad l(10y)=00y
\end{gather*}
for $y\in\{0,1\}^{\N}$ (see Figure \ref{fig:pl}). Then the subgroup of $G$ generated by $p$ and $l$ is dense in $G$. In fact, this subgroup is isomorphic to the Thompson's group $V$.
\end{theorem}

Here we used the identification of $E$ with $\{0,1\}^{\N}$ described in Section \ref{sec:prelim}. More details on the Thompson group can be found in Section \ref{sec:gen}. It has been known that the Thompson's group $V$ is generated by two elements \cite{BQ}.
%and generating sets consisting of elements of finite order are found in \cite{BEH} and \cite{SSW}.
The new information in Theorem \ref{thm:gen} is that one generator can be taken in $P$ and the other one in $L$, and that both can be chosen to be very simple.

\begin{figure}[ht]
\centering
\begin{tikzpicture}[x=0.92cm,y=0.92cm,line width=0.4pt]
% positions of the four quarters
\def\qa{0}\def\qb{1.2}\def\qc{2.6}\def\qd{3.8}
\newcommand{\panel}[1]{%
  % the Cantor set along both axes, represented by four intervals
  \foreach \x/\w in {\qa/00,\qb/01,\qc/10,\qd/11}{
    \draw (\x,-0.3)--++(1,0);
    \node[font=\scriptsize,below] at (\x+0.5,-0.33) {$\w$};
    \draw (-0.3,\x)--++(0,1);
    \node[font=\scriptsize,left] at (-0.33,\x+0.5) {$\w$};
  }
  \node at (2.4,-1.25) {#1};
}
% ---------- p ----------
\begin{scope}
\panel{$p$}
\draw (\qa,\qa)--++(1,0.5);
\draw (\qb,\qa+0.5)--++(1,0.5);
\draw (\qc,\qb)--++(1,1);
\draw (\qd,\qc)--++(0.5,1);
\draw (\qd+0.5,\qd)--++(0.5,1);
\end{scope}
% ---------- l ----------
\begin{scope}[xshift=6.6cm]
\panel{$l$}
\draw (\qa,\qb)--++(1,1);
\draw (\qb,\qd)--++(1,1);
\draw (\qd,\qc)--++(1,1);
\draw (\qc,\qa)--++(1,1);
\end{scope}
\end{tikzpicture}
\caption{The graphs of the generators $p$ and $l$ from Theorem \ref{thm:gen}.}\label{fig:pl}
\end{figure}
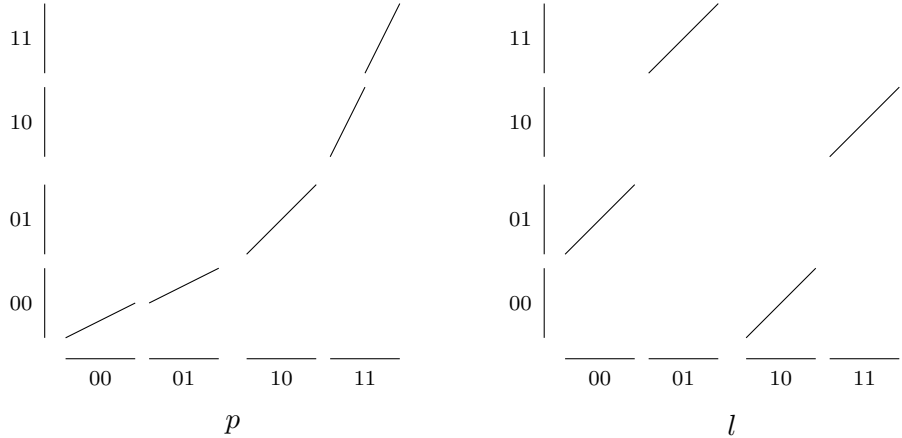
%The set of all pairs in $P\times L$ which generate a dense subgroup is not dense in $P\times L$ (Remark \ref{rem:gen}), so the existence of such a pair does not follow from a routine genericity argument.

In Section \ref{sec:remarks} we show that the generic homeomorphism of $E$ is the product of a generic element of $P$ and a generic element of $L$, although not every such product is generic (Proposition \ref{prop:generic}).
We also observe that the equality $G=PL$ fails if $\lambda$ is replaced by an arbitrary full nonatomic measure on $E$ (Remark \ref{rem:anymeasure}). We do not know whether $P$ or $L$ has a closed complement in $G$ (Question \ref{q:complement}).

The paper is organized as follows. Section \ref{sec:prelim} contains the preliminaries, in particular
Akin's classification of measures on an ordered Cantor space up to order isomorphism and the
resulting criterion for membership in $PL$. Section \ref{sec:example} contains an example of a homeomorphism which is not in $PL$.
In Section \ref{sec:bigcell} we prove Theorem \ref{thm:main}. In Section \ref{sec:GW} we explain what
remains true of the other statements of \cite[Section 2]{GW} which depend on \cite[Theorem 2.1]{GW}.
Theorem \ref{thm:cosets} is proved in Section \ref{sec:cosets} and Theorem \ref{thm:gen} in Section \ref{sec:gen}. Section \ref{sec:remarks} contains further results and questions.

Let us stress that the main theorem of \cite{GW} is independent of
\cite[Theorem 2.1]{GW}. We also note that another inaccuracy in \cite{GW}, at the end of
Section 3 there, was pointed out and repaired by Glasner and Gutman in \cite{GG}.

\section{Preliminaries}\label{sec:prelim}

\subsection{Coding}
We identify $E$ with $\cs=\{0,1\}^{\N}$ by means of the map
$x\mapsto\sum_{i\ge1}2x_i3^{-i}$. This map is a homeomorphism, it is an isomorphism of the
lexicographic order of $\cs$ onto the order of $E$ inherited from the real line, and it carries
the product of the measures $(\frac12,\frac12)$ to the Cantor--Lebesgue measure $\lambda$.
For a finite word $w\in\{0,1\}^n$ we write $|w|=n$ and $[w]=\{x\in E: x\text{ extends }w\}$.
The sets $[w]$ with $|w|=n$ are precisely the basic intervals $I^n_j$, $1\le j\le 2^n$, of \cite{GW};
they are clopen, convex, pairwise disjoint, they have diameter $3^{-n}$ and $\lambda([w])=2^{-n}$.
For a word $w$ and $y\in\cs$ (or a word $y$), $wy$ denotes the concatenation, and $\zinf$, $\oinf$ are the constant
sequences. The least and the greatest elements of $E$ are $0=\zinf$ and $1=\oinf$.

The gaps of $E$ are the pairs $(w0\oinf,w1\zinf)$, where $w$ is a finite word. We denote by
$E^-=\{w0\oinf\}$ the set of left endpoints and by $E^+=\{w1\zinf\}$ the set of right endpoints of
gaps. These sets are disjoint and both are dense in $E$. For $t\in E$ we put $E_t=\{x\in E:x\le t\}$. Then
\[
\cJ=\{E_t:t\in E^-\}
\]
is the family of all nonempty proper clopen initial segments of $E$, and we let
$\cJ_0=\cJ\cup\{\emptyset,E\}$. For every $n$, the unions of the first $i$ of the sets $[w]$, $|w|=n$ (in the increasing order),
$0\le i\le 2^n$, belong to $\cJ_0$, and every element of $\cJ_0$ is of this form for some $n$. Consequently:

\begin{enumerate}[label=\textup{(E\arabic*)}]
\item\label{clopA} $\lambda$ restricted to $\cJ$ is a strictly increasing bijection of $(\cJ,\subseteq)$ onto $\D$, the set of dyadic
rationals in $(0,1)$. In particular $\St(\lambda)=\D$.
\item\label{clopB} Every nonempty clopen set $W\subseteq E$ can be written as a disjoint union
$W=\bigcup_{i=1}^k (J_i'\setminus J_i)$ with $J_1\subsetneq J_1'\subseteq J_2\subsetneq J_2'\subseteq\dots\subseteq J_k\subsetneq J_k'$ in $\cJ_0$.
\item\label{clopC} If $C\subseteq E$ is clopen and $r$ is a dyadic rational with $0\le r\le\lambda(C)$, then there is a clopen
set $C'\subseteq C$ with $\lambda(C')=r$. 
%(Choose $n$ so large that $C$ is a union of sets $[w]$ with $|w|=n$ and $2^nr$ is an integer, and take $2^nr$ of these sets.)
\end{enumerate}

\subsection{The groups}
Let $G=\Homeo(E)$ be equipped with the topology of uniform convergence, $P=\{g\in G: g\text{ is order preserving}\}$ and
$L=\{g\in G:g_*\lambda=\lambda\}$.
Every $p\in P$ maps gaps onto gaps, hence $p(E^-)=E^-$, $p(E^+)=E^+$ and
\begin{equation}\label{eq:PJ}
\{p(J):J\in\cJ\}=\cJ\qquad\text{for every }p\in P.
\end{equation}

We omit the proof of following lemma.

\begin{lemma}\label{lem:top}

\begin{enumerate}[label=\textup{(\alph*)}]
\item\label{topA} For any clopen sets $A,B\subseteq E$ the set $\{g\in G:g(A)=B\}$ is clopen in $G$.
\item\label{topB} For $g\in G$ and $n\ge1$ let
\[
N_n(g)=\{h\in G: h^{-1}[w]=g^{-1}[w]\text{ for every word $w$ with }|w|=n\}.
\]
Then $\{N_n(g):n\ge1\}$ is a base of neighbourhoods of $g$ in $G$.
\item\label{topC} $P$ and $L$ are closed subgroups of $G$.
\end{enumerate}
\end{lemma}

\subsection{Measures on the ordered Cantor set}
By a \emph{measure} we mean a Borel probability measure on $E$. A measure is \emph{full} if every nonempty open set has positive measure.
The \emph{clopen values set} of $\mu$ is $S(\mu)=\{\mu(A):A\subseteq E\text{ clopen}\}$, and the special clopen values set
$\St(\mu)=\{\mu(J):J\in\cJ\}$ was defined in the introduction. In \cite{Ak1} the numbers $0$ and $1$ are included in $\St(\mu)$, however, this makes no difference.
For $g\in G$ the measure $g_*\mu$ is defined by $g_*\mu(B)=\mu(g^{-1}B)$; clearly the measure $g_*\mu$ is full and nonatomic if $\mu$ is, and
\begin{equation}\label{eq:Sinv}
S(g_*\mu)=S(\mu)\ \text{ for } g\in G,\qquad \St(p_*\mu)=\St(\mu)\ \text{ for } p\in P,
\end{equation}
the latter by \eqref{eq:PJ}. Note that $S(\lambda)=\D\cup\{0,1\}$.

The \emph{distribution function} of a measure $\mu$ is $F_\mu\colon E\to[0,1]$, $F_\mu(t)=\mu(E_t)$.

\begin{lemma}\label{lem:F}
Let $\mu$ be a full nonatomic measure. Then $F_\mu$ is a continuous nondecreasing map of $E$ onto $[0,1]$, and for $s<t$ in $E$ we have
$F_\mu(s)=F_\mu(t)$ if and only if $(s,t)$ is a gap. Consequently:
\begin{enumerate}[label=\textup{(\alph*)}]
\item\label{FA} for $c\in[0,1]$, the fibre $F_\mu^{-1}(c)$ consists of the two endpoints of a gap if $c\in\St(\mu)$, and it is a singleton otherwise;
\item\label{FB} $\St(\mu)$ is a countable dense subset of $(0,1)$, and $J\mapsto\mu(J)$ is a strictly increasing bijection of $\cJ$ onto $\St(\mu)$.
\end{enumerate}
\end{lemma}

\begin{proof}
Clearly $F_\mu$ is nondecreasing, and it is continuous since $\mu$ is nonatomic. For $s<t$ we have $F_\mu(t)-F_\mu(s)=\mu(E\cap(s,t])$ since the $\mu$-measure of singletons is zero. So $F_\mu(s)<F_\mu(t)$ iff $(s,t)$ is not a gap.
Further, $F_\mu(0)=\mu(\{0\})=0$ and $F_\mu(1)=1$. If some $c\in(0,1)$ were not a value of $F_\mu$, we could put
$s=\max F_\mu^{-1}[0,c]$ and $t=\min F_\mu^{-1}[c,1]$; then $s<t$, $E\cap(s,t)=\emptyset$ and $F_\mu(s)<c<F_\mu(t)$, a contradiction.

\ref{FA} By the first part, every fibre is a singleton or a pair of endpoints of a gap, and the fibre of $c$ is of the second kind
if and only if it contains a point $s\in E^-$, that is, if and only if $c=\mu(E_s)$ for some $s\in E^-$.

\ref{FB} If $s<t$ are in $E^-$, then $(s,t)$ is not a gap, because $t\notin E^+$. Hence, by the first part, $F_\mu$ is strictly increasing on $E^-$. As $0$ and $1$ are not endpoints of gaps, $F_\mu(s)\in(0,1)$ for $s\in E^-$.
Finally, $\St(\mu)=F_\mu(E^-)$ is dense in $F_\mu(E)=[0,1]$ because $E^-$ is dense in $E$.
\end{proof}

The following result is due to Akin \cite[Theorem 2.10]{Ak1}.

\begin{proposition}[Akin]\label{prop:akin}
Let $\mu$ and $\nu$ be full nonatomic measures on $E$. There exists $p\in P$ with $p_*\mu=\nu$ if and only if $\St(\mu)=\St(\nu)$.
Moreover, such $p$ is unique.
\end{proposition}

\begin{corollary}\label{cor:criterion}
\begin{enumerate}[label=\textup{(\alph*)}]
\item\label{critA} $P\cap L=\{\id\}$.
\item\label{critB} For every $g\in G$, the measure $g_*\lambda$ is full and nonatomic and $\St(g_*\lambda)\subseteq\D$.
\item\label{critC} For $g\in G$ we have $g\in PL$ if and only if $\St(g_*\lambda)=\D$, if and only if for every dyadic rational $d\in(0,1)$ there is
$J\in\cJ$ with $\lambda(g^{-1}J)=d$. Consequently, $g\in LP$ if and only if $\St(g^{-1}_*\lambda)=\D$.
\item\label{critD} For $g,g'\in G$ we have $PgL=Pg'L$ if and only if $\St(g_*\lambda)=\St(g'_*\lambda)$.
\end{enumerate}
\end{corollary}

\begin{proof}
\ref{critA} follows from the uniqueness in Proposition \ref{prop:akin} applied to $\mu=\nu=\lambda$.
The first part of \ref{critB} was noted above. For the second part, by \eqref{eq:Sinv}, $\St(g_*\lambda)\subseteq S(g_*\lambda)\setminus\{0,1\}=S(\lambda)\setminus\{0,1\}=\D$.
\ref{critD} If $g'=pgl$ with $p\in P$, $l\in L$, then $g'_*\lambda=p_*(g_*\lambda)$ and by \eqref{eq:Sinv} 
\[\St(g'_*\lambda)=\St(p_*(g_*\lambda))=\St(g_*\lambda).\]
Conversely, if $\St(g_*\lambda)=\St(g'_*\lambda)$,
then Proposition \ref{prop:akin} gives $p\in P$ with $p_*g_*\lambda=g'_*\lambda$, so $l=g'^{-1}pg\in L$ and $g'=pgl^{-1}$.
\ref{critC} is the special case $g'=\id$ of \ref{critD}, because $\St(\lambda)=\D$. The last statement of \ref{critC} follows from $LP=(PL)^{-1}$.
\end{proof}

\begin{remark}\label{rem:gap}
We can now locate the gap in the proof of \cite[Theorem 2.1]{GW}. Given $g\in G$ and $\nu=g_*\lambda$, a map $p$ is defined there on points $t$ for which $E_t$ is clopen by the requirement that
$E_{p(t)}$ be clopen and $\lambda(E_t)=\nu(E_{p(t)})$, and $p$ is then extended to a homeomorphism of $E$ with $p_*\lambda=\nu$. 
However, for $p(t)$ to exist for every $t\in E^-$ one needs the inclusion $\D\subseteq\St(\nu)$.
%Once that inclusion holds, the equality $\St(\nu)=\D$ yields that the argument of \cite{GW} is correct.
\end{remark}

\section{A homeomorphism which is not in \texorpdfstring{$PL$}{PL}}\label{sec:example}

Let
\[
\xs=010101\ldots\in E ,
\]
and for $n\ge1$ let $u_n$ be the word of length $n$ which agrees with $\xs$ at the first $n-1$ places and differs from $\xs$ at the $n$-th place.
Further, for $j\ge1$ let
\[
v_{2j-1}=0^{j-1}11,\qquad v_{2j}=0^{j-1}10 .
\]
Thus
\[
\begin{array}{c|cccccc}
n & 1&2&3&4&5&6\\\hline
u_n & 1 & 00 & 011 & 0100 & 01011 & 010100\\
v_n & 11 & 10 & 011 & 010 & 0011 & 0010
\end{array}
\]
Put $K_n=[u_n]$ and $K_n'=[v_n]$. The sets $K_n$, $n\ge1$, form a partition of $E\setminus\{\xs\}$ into clopen sets, and the sets $K_n'$, $n\ge1$,
form a partition of $E\setminus\{\zinf\}$ into clopen sets. We define $g_0\colon E\to E$ by
\begin{equation}\label{eq:g0}
g_0(v_ny)=u_ny\quad (n\ge1,\ y\in\cs),\qquad g_0(\zinf)=\xs .
\end{equation}

\begin{proposition}\label{prop:example}
The map $g_0$ is a homeomorphism of $E$, and
\begin{enumerate}[label=\textup{(\alph*)}]
\item\label{exA} $\St((g_0)_*\lambda)=\D\setminus\{\frac12\}$,
\item\label{exB} $\St((g_0^{-1})_*\lambda)=\D$.
\end{enumerate}
Consequently $g_0\in LP\setminus PL$ and $g_0^{-1}\in PL\setminus LP$.
\end{proposition}

For the proof we need a simple observation.

\begin{lemma}\label{lem:local}
Let $\mu$ be a full nonatomic measure, let $u$ be a word and assume that there is a constant $c>0$ such that $\mu(B)=c\lambda(B)$ for every Borel set $B\subseteq[u]$.
Let $a=\mu(\{x\in E:x<u\zinf\})\in\D\cup\{0\}$ and $m=\mu([u])=2^{-k}$ for some integer $k\ge0$. Then
\[
\{\mu(J): J\in\cJ,\ \max J\in[u]\setminus\{u\oinf\}\}
\]
is the set of all dyadic rationals in the open interval $(a,a+m)$.
\end{lemma}

\begin{proof}
The sets $J\in\cJ$ with $\max J\in[u]\setminus\{u\oinf\}$ are exactly the sets $E_{ut}$ with $t\in E^-$, and
\[
\mu(E_{ut})=a+c\lambda([u]\cap E_{ut})=a+c2^{-|u|}\lambda(E_t)=a+m\lambda(E_t).
\]
The numbers $\lambda(E_t)$, $t\in E^-$, are exactly the elements of $\D$. Finally, $a+m\D$ is the set of all dyadic rationals in $(a,a+m)$,
because $a$ is dyadic and $m=2^{-k}$.
%, so that a number $q\in(a,a+m)$ is dyadic if and only if $(q-a)/m$ is.
\end{proof}

\begin{proof}[Proof of Proposition \ref{prop:example}]
The map $g_0$ is a bijection, and it is continuous on each of the clopen sets $K_n'$. Since $K_n'\to 0$, $g_0(K_n')=K_n$ and $K_n\to \xs=g_0(0)$, $g_0$ is continuous at 0. Thus $g_0$ is a continuous bijection of a compact space, hence a homeomorphism.

\ref{exA} Let $\nu=(g_0)_*\lambda$. By \eqref{eq:g0}, for every Borel set $B\subseteq K_n$ we have $\nu(B)=\lambda(g_0^{-1}B)=2^{|u_n|-|v_n|}\lambda(B)$, and
\begin{equation}\label{eq:nuK}
\nu(K_{2j-1})=\nu(K_{2j})=\lambda([v_{2j}])=2^{-(j+1)}\qquad(j\ge1).
\end{equation}
The $n$-th digit of $\xs$ is $1$ if and only if $n$ is even. Hence $K_n$ lies below $\xs$ for $n$ even and above $\xs$ for $n$ odd. Moreover
\[
K_2<K_4<K_6<\dots<\xs<\dots<K_5<K_3<K_1 ,
\]
where $A<B$ means that every element of $A$ is smaller than every element of $B$. For $a_n=\nu(\{x:x<u_n\zinf\})$ we obtain from \eqref{eq:nuK}
\[
a_{2j}=\sum_{i<j}2^{-(i+1)}=\tfrac12-2^{-j},\qquad a_{2j-1}=1-\sum_{i\le j}2^{-(i+1)}=\tfrac12+2^{-(j+1)} .
\]
Every $J\in\cJ$ has its maximum in $E^-$, hence in some $K_n$, because $\xs\notin E^-$. By Lemma \ref{lem:local} and \eqref{eq:nuK}, the values $\nu(J)$ for
$J\in\cJ$ with $\max J\in K_{2j}$ are precisely the dyadic rationals in the interval $(\frac12-2^{-j},\frac12-2^{-(j+1)}]$ (the right endpoint corresponds to $\max J=\max K_{2j}$),
and the values $\nu(J)$ for $J\in\cJ$ with $\max J\in K_{2j-1}$ are precisely the dyadic rationals in $(\frac12+2^{-(j+1)},\frac12+2^{-j}]$ different from $1$.
The union of these intervals over $j\ge1$ is $(0,\frac12)\cup(\frac12,1]$. Hence $\St(\nu)=\D\setminus\{\frac12\}$.

\ref{exB} Let $\mu=(g_0^{-1})_*\lambda$. For every Borel set $B\subseteq K_n'$ we have $\mu(B)=\lambda(g_0B)=2^{|v_n|-|u_n|}\lambda(B)$ and $\mu(K_n')=\lambda(K_n)=2^{-n}$.
Moreover $\zinf<\dots<K_3'<K_2'<K_1'$, so $\mu(\{x:x<v_n\zinf\})=\sum_{i>n}2^{-i}=2^{-n}$. By Lemma \ref{lem:local}, the values $\mu(J)$ for $J\in\cJ$ with $\max J\in K_n'$ are exactly
the dyadic rationals in $(2^{-n},2^{1-n}]$ different from $1$. The union of these intervals is $(0,1]$, so $\St(\mu)=\D$.

By \ref{exA}, \ref{exB} and Corollary \ref{cor:criterion}\ref{critC}, $g_0\in LP\setminus PL$. Thus also $g_0^{-1}\in PL\setminus LP$.
\end{proof}

\section{Proof of Theorem \ref{thm:main}}\label{sec:bigcell}

Part \ref{main1} is contained in Lemma \ref{lem:top}\ref{topC} and Corollary \ref{cor:criterion}\ref{critA}; if $pl=p'l'$ then $p'^{-1}p=l'l^{-1}\in P\cap L$.

Part \ref{main2}. By Proposition \ref{prop:example}, $g_0\in LP\setminus PL$. If $PL$ were a subgroup, we would have $PL=(PL)^{-1}=LP$. Similarly for $LP$.

Part \ref{main3}. For $d\in\D$ let
\begin{align*}
U_d&=\{g\in G:\ \lambda(g^{-1}J)=d\text{ for some }J\in\cJ\}\\
&=\bigcup\bigl\{\{g\in G:g(B)=J\}:J\in\cJ,\ B\subseteq E\text{ clopen},\ \lambda(B)=d\bigr\}.
\end{align*}
This set is open by Lemma \ref{lem:top}\ref{topA}, and $PL=\bigcap_{d\in\D}U_d$ by Corollary \ref{cor:criterion}\ref{critC}. Hence $PL$ is a $G_\delta$ set.
By the Baire category theorem, to prove that $PL$ is dense it suffices to show that each $U_d$ is dense.

Fix $d\in\D$, $g\in G$ and $n\ge1$; by Lemma \ref{lem:top}\ref{topB} we have to find $h\in N_n(g)\cap U_d$. Let $A_1<A_2<\dots<A_m$, $m=2^n$, be the sets $[w]$, $|w|=n$, in the increasing order,
and let $C_i=g^{-1}(A_i)$, $c_i=\lambda(C_i)$ and $\sigma_i=c_1+\dots+c_i$. The numbers $c_i$ are positive dyadic rationals and $\sigma_m=1$. If $d=\sigma_i$ for some $i<m$, then
$J=A_1\cup\dots\cup A_i\in\cJ$ and $\lambda(g^{-1}J)=d$, so $g\in U_d$. Otherwise there is $i$ with $\sigma_{i-1}<d<\sigma_i$ (where $\sigma_0=0$). Then $r=d-\sigma_{i-1}$ is a dyadic rational with $0<r<c_i$, and by
\ref{clopC} there is a clopen set $C'\subseteq C_i$ with $\lambda(C')=r$. Let $A'=[w0]$ where $A_i=[w]$. All four sets $C'$, $C_i\setminus C'$, $A'$, $A_i\setminus A'$ are nonempty clopen subsets
of $E$, hence they are homeomorphic to $E$. Let $h\in G$ be such that $h=g$ on $E\setminus C_i$, $h(C')=A'$ and $h(C_i\setminus C')=A_i\setminus A'$. Then $h^{-1}(A_k)=C_k$ for all $k$, so $h\in N_n(g)$, and
$J=A_1\cup\dots\cup A_{i-1}\cup A'\in\cJ$ satisfies $\lambda(h^{-1}J)=\sigma_{i-1}+r=d$. Hence $h\in U_d$.

So $PL$ is a dense $G_\delta$ set, and so is $LP=(PL)^{-1}$, since the inversion is a homeomorphism of $G$.

The multiplication map $\varphi\colon P\times L\to PL$, $\varphi(p,l)=pl$, is continuous, and it is a bijection by \ref{main1}. For $g\in PL$ let $g=p_gl_g$ be the unique decomposition. To prove that $\varphi^{-1}$ is continuous,
let $g_k\to g$ in $PL$ and fix $n\ge1$. As above, let $A_1<\dots<A_m$ be the sets $[w]$, $|w|=n$, and let $J_i=A_1\cup\dots\cup A_i\in\cJ$ for $1\le i<m$.
Since $g_{k*}\lambda=(p_{g_k})_*\lambda$, we have
\[
\lambda\bigl(p_{g_k}^{-1}J_i\bigr)=\lambda\bigl(g_k^{-1}J_i\bigr),
\]
and the right-hand side equals $\lambda(g^{-1}J_i)=\lambda(p_g^{-1}J_i)$ as soon as $g_k\in N_n(g)$, that is, for all sufficiently large $k$. Both $p_{g_k}^{-1}J_i$ and $p_g^{-1}J_i$ belong to $\cJ$ by \eqref{eq:PJ},
and $\lambda$ is injective on $\cJ$. Hence, for all sufficiently large $k$, $p_{g_k}^{-1}J_i=p_g^{-1}J_i$ for every $i<m$, so $p_{g_k}^{-1}A_i=p_g^{-1}A_i$ for every $i\le m$, that is, $p_{g_k}\in N_n(p_g)$.
By Lemma \ref{lem:top}\ref{topB}, $p_{g_k}\to p_g$, and therefore also $l_{g_k}=p_{g_k}^{-1}g_k\to p_g^{-1}g=l_g$. The statement about $L\times P\to LP$ follows by taking inverses.

Part \ref{main4}. Let $g\in G$. The sets $PL$ and $g\cdot PL$ are comeager in the Polish group $G$ by Part \ref{main3}, so they intersect: $gpl=p'l'$ for some $p,p'\in P$ and $l,l'\in L$.
Then $g=p'l'l^{-1}\,p^{-1}\in PLP$. Similarly, $g\in LPL$. \qed

\begin{corollary}\label{cor:conj}
Every homeomorphism of $E$ is conjugate to an element of $PL$. More precisely, for every $g\in G$ there are $q\in P$ and $h\in L$ such that both $qgq^{-1}$ and $hgh^{-1}$ belong to $PL$. The same is true for $LP$ in place of $PL$.
\end{corollary}

\begin{proof}
By Theorem \ref{thm:main}\ref{main4} we can write $g=p_1lp_2$ with $p_1,p_2\in P$ and $l\in L$. Then $p_2gp_2^{-1}=(p_2p_1)l\in PL$ and $p_1^{-1}gp_1=l(p_2p_1)\in LP$. Similarly, if $g=l_1pl_2$ with $l_1,l_2\in L$ and $p\in P$, then
$l_1^{-1}gl_1=p(l_2l_1)\in PL$ and $l_2gl_2^{-1}=(l_2l_1)p\in LP$.
\end{proof}

\begin{remark}\label{rem:interior}
The set $PL$ has empty interior, so both $PL$ and its complement are dense in $G$. Indeed, let $g\in G$ and $n\ge1$, let $w=0^n$ be the word of lenght $n$ consisting of zeros. Choose a word $z$ with
$[z]\subsetneq g^{-1}[w]$. Let $h\in G$ be such that $h=g$ outside $g^{-1}[w]$, $h(zy)=w0g_0(y)$ for $y\in\cs$, where $g_0$ is the homeomorphism from Section \ref{sec:example}, and $h$ maps
$g^{-1}[w]\setminus[z]$ homeomorphically onto $[w1]$. Then $h\in N_n(g)$. Since $[w]$ is an initial segment of $E$, we have $E_{w0\xs}\subseteq[w0]$ and $h^{-1}(E_{w0\xs})=\{zy: y\in g_0^{-1}(E_{\xs})\}$.
This set has $\lambda$-measure $2^{-|z|-1}$. As $w0\xs$ is not an endpoint of a gap, $2^{-|z|-1}\notin\St(h_*\lambda)$ by Lemma \ref{lem:F}. Hence $h\in N_n(g)\setminus PL$.
\end{remark}

\section{Consequences for Section 2 of \texorpdfstring{\cite{GW}}{[GW]}}\label{sec:GW}

Following \cite{GW}, let $H_n=\{g\in G: g[w]=[w]\text{ for every word $w$ with }|w|=n\}$, $L_n=L\cap H_n$ and $P_n=P\cap H_n$. The groups $H_n$ form a base of neighbourhoods of the identity consisting of clopen subgroups.
After the proof of \cite[Theorem 2.1]{GW} it is stated that $H_n=L_nP_n=P_nL_n$. What is true is that
\begin{equation}\label{eq:Hn}
H_n\cap PL=P_nL_n\qquad\text{and}\qquad H_n\cap LP=L_nP_n .
\end{equation}
Indeed, let $g=pl\in H_n$ with $p\in P$, $l\in L$, and let $J$ be the union of the first $i$ of the sets $[w]$, $|w|=n$, where $1\le i<2^n$. Then $\lambda(p^{-1}J)=\lambda(g^{-1}J)=\lambda(J)$, and since $p^{-1}J\in\cJ$ and $\lambda$ is injective on $\cJ$,
we get $p^{-1}J=J$. Hence $p\in H_n$ and $l=p^{-1}g\in H_n$. The second equality in \eqref{eq:Hn} follows by taking inverses. However, the equalities $H_n=P_nL_n$ and $H_n=L_nP_n$ fail:

\begin{proposition}\label{prop:Hn}
For every $n\ge1$ we have $H_n\not\subseteq PL$ and $H_n\not\subseteq LP$.
\end{proposition}

\begin{proof}
Let $g_0$ be the homeomorphism from Section \ref{sec:example} and define $g_n\in H_n$ by $g_n(0^ny)=0^ng_0(y)$ for $y\in\cs$ and $g_n(x)=x$ for $x\notin[0^n]$.
Then $g_n^{-1}(E_{0^n\xs})=\{0^ny:y\in g_0^{-1}(E_{\xs})\}$ has $\lambda$-measure $2^{-n}F_{(g_0)_*\lambda}(\xs)=2^{-n-1}$. Since $0^n\xs$ is not an endpoint of a gap, Lemma \ref{lem:F} shows that
$2^{-n-1}\notin\St(g_{n*}\lambda)$. Hence $g_n\notin PL$ and $g_n^{-1}\in H_n\setminus LP$.
\end{proof}

Next we turn to \cite[Proposition 2.2]{GW}. Let $\pi\colon G\to G/P$ be the quotient map onto the space of left cosets $gP$. Recall that $\Phi=\Phi(E)$ is the compact space of all maximal chains of closed subsets of $E$ and that
$c_0=\{E_t:t\in E\}\in\Phi$ is the chain of all nonempty closed initial segments; the stabilizer of $c_0$ in $G$ is $P$ \cite[Theorem 1.3]{GW}.
Proposition 2.2 of \cite{GW} states that (i) $\pi$ restricts to a homeomorphism of $L$ onto $G/P$, and (ii) the map $gP\mapsto gc_0$ of $G/P$ onto the orbit $Gc_0\subseteq\Phi$ is one-to-one and continuous but not a homeomorphism.
The proof of (i) uses the invalied statement \cite[Theorem 2.1]{GW}, and the proof of (ii) uses (i). The correct version is as follows.

\begin{proposition}\label{prop:22}
\begin{enumerate}[label=\textup{(\roman*)}]
\item\label{22i} The restriction of $\pi$ to $L$ is a homeomorphism of $L$ onto $\pi(L)=\pi(LP)$, which is a dense $G_\delta$ subset of $G/P$ different from $G/P$.
\item\label{22ii} The map $\beta\colon G/P\to Gc_0$, $\beta(gP)=gc_0$, is a continuous bijection which is not a homeomorphism.
\end{enumerate}
\end{proposition}

\begin{proof}
\ref{22i} Recall that $\pi$ is continuous and open. The set $LP$ is saturated, i.e., $LP=\pi^{-1}(\pi(LP))$, and $LP=\bigcap_{d\in\D}U_d^{-1}$, where the open sets $U_d^{-1}=\{g^{-1}:g\in U_d\}$ are dense and also saturated (note that $PU_d=U_d$ by \eqref{eq:PJ}).
Hence $\pi(LP)=\bigcap_{d}\pi(U_d^{-1})$ is a dense $G_\delta$ subset of $G/P$, and $\pi(g_0^{-1})\notin\pi(LP)$ by Proposition \ref{prop:example}. The restriction $\pi|_L$ is injective because $L\cap P=\{\id\}$.
By Theorem \ref{thm:main}\ref{main3}, the map $q\colon LP\to L$, $q(lp)=l$, is continuous. It is constant on the cosets $gP\subseteq LP$, and the restriction of $\pi$ to the saturated set $LP$ is an open map onto $\pi(LP)$.
Hence $q$ induces a continuous map $\bar q\colon\pi(LP)\to L$ with $\bar q\circ\pi=q$, and $\bar q$ is the inverse of $\pi|_L$.

\ref{22ii} The map $\beta$ is a bijection because $P$ is the stabilizer of $c_0$, and it is continuous because $g\mapsto gc_0$ is. For $n\ge1$ let $l_n\in G$ be the involution which exchanges the sets $[01^n]$ and $[10^{n-1}1]$,
\[
l_n(01^ny)=10^{n-1}1y,\qquad l_n(10^{n-1}1y)=01^ny\qquad(y\in\cs),
\]
and which is the identity elsewhere. Then $l_n\in L$ and $l_n\notin H_1$, so $l_n\not\to\id$ in $L$, and by \ref{22i} the cosets $l_nP$ do not converge to $P$ in $G/P$. We claim that nevertheless $\beta(l_nP)=l_nc_0\to c_0=\beta(P)$ in $\Phi$.
For $m\ge1$ let $A_1<\dots<A_{2^m}$ be the sets $[w]$, $|w|=m$. By \cite[Lemma 1.1]{GW}, a base of neighbourhoods of $c_0$ in $\Phi$ is formed by the sets $\mathcal U^m$, $m\ge1$ (denoted by $U^m_{\id}$ in \cite{GW}), where $c\in\mathcal U^m$ if and only if
the sets $\{i:F\cap A_i\neq\emptyset\}$, $F\in c$, are exactly the initial segments $\{1,\dots,j\}$, $1\le j\le 2^m$. For $l\in G$, the member $l(E_t)$ of the chain $lc_0$ meets $A_i$ if and only if $t\ge\min l^{-1}(A_i)$. Hence
\[
lc_0\in\mathcal U^m\quad\Longleftrightarrow\quad \min l^{-1}(A_1)<\min l^{-1}(A_2)<\dots<\min l^{-1}(A_{2^m}).
\]
Let $n\ge m$. Then $l_n^{-1}(A_i)=l_n(A_i)=A_i$ unless $A_i$ is one of the two adjacent sets $[01^{m-1}]$, $[10^{m-1}]$. Further, $\min l_n[01^{m-1}]=01^{m-1}\zinf=\min\,[01^{m-1}]$ (this point does not belong to $[01^n]$ because $n\ge m$), while
$\min l_n[10^{m-1}]=01^n\zinf$, which is greater than $01^{m-1}\zinf$ and smaller than every element of $[1]$. Thus the minima are still increasing and $l_nc_0\in\mathcal U^m$ for all $n\ge m$.
\end{proof}

\section{Double cosets}\label{sec:cosets}

In this section we prove Theorem \ref{thm:cosets}. We need the notion of a good measure, which was introduced by Akin \cite{Ak2}. A full nonatomic measure $\mu$ on a Cantor space $X$ is called \emph{good} if for any clopen sets $U,V\subseteq X$ with $\mu(U)<\mu(V)$ there exists a clopen set
$W\subseteq V$ with $\mu(W)=\mu(U)$. In other words, every value which $\mu$ takes on clopen sets and which is smaller than $\mu(V)$ is already taken on some clopen subset of $V$.
Good measures are exactly the invariant measures of uniquely ergodic minimal homeomorphisms of Cantor spaces \cite{Ak2}. By \ref{clopC}, the Cantor--Lebesgue measure $\lambda$ is good, and so is $g_*\lambda$ for every $g\in G$.
The importance of good measures comes from the fact that they are classified up to homeomorphism by their clopen values sets. We shall use this theorem of Akin \cite[Theorem 2.9]{Ak2} in the following form:
\emph{if $\mu$ and $\nu$ are good measures on $E$ and $S(\mu)=S(\nu)$, then $\nu=h_*\mu$ for some $h\in G$.}

For a full nonatomic measure $\nu$ and $c\in\St(\nu)\cup\{0,1\}$ let $J_c$ denote the unique element of $\cJ_0$ with $\nu(J_c)=c$ (see Lemma \ref{lem:F}\ref{FB}). By \ref{clopB}, the nonempty clopen subsets of $E$ are exactly the sets
\begin{equation}\label{eq:W}
W=\bigcup_{i=1}^k(J_{b_i}\setminus J_{a_i}),\qquad a_1<b_1\le a_2<b_2\le\dots\le a_k<b_k\ \text{ in }\St(\nu)\cup\{0,1\},
\end{equation}
we have $\nu(W)=\sum_{i}(b_i-a_i)$, and $W\subseteq J_b\setminus J_a$ if and only if $a\le a_1$ and $b_k\le b$. Therefore, with the notation of the introduction and $S=\St(\nu)$,
\begin{equation}\label{eq:Sigma}
\{\nu(W): W\subseteq J_b\setminus J_a\text{ nonempty clopen}\}=\Sigma_S(a,b)\qquad(a<b\text{ in }S\cup\{0,1\}).
\end{equation}

\begin{lemma}\label{lem:adm}
Let $\nu$ be a full nonatomic measure on $E$. Then $\nu=h_*\lambda$ for some $h\in G$ if and only if $\St(\nu)$ is admissible.
\end{lemma}

\begin{proof}
Let $S=\St(\nu)$. Assume that $\nu=h_*\lambda$. Then $S\subseteq\D$ by Corollary \ref{cor:criterion}\ref{critB}, and $S$ is dense by Lemma \ref{lem:F}\ref{FB}. Let $a<b$ in $S\cup\{0,1\}$ and let $r\in\D$, $r<b-a$.
Then $r\in S(\lambda)=S(\nu)$ and $r<\nu(J_b\setminus J_a)$. As $\nu$ is good, there is a clopen set $W\subseteq J_b\setminus J_a$ with $\nu(W)=r$, and $r\in\Sigma_S(a,b)$ by \eqref{eq:Sigma}.

Conversely, let $S$ be admissible. By \eqref{eq:Sigma} with $a=0$, $b=1$ we have $S(\nu)=\Sigma_S(0,1)\cup\{0\}$. Since $S\subseteq\D$, this set is contained in $\D\cup\{0,1\}$, and it contains $\D$ by \eqref{eq:adm}. Hence $S(\nu)=\D\cup\{0,1\}=S(\lambda)$.
We show that $\nu$ is good. Let $U,V$ be clopen with $\nu(U)<\nu(V)$ and write $V$ in the form \eqref{eq:W}. Put $r=\nu(U)$; this is a dyadic rational. Define $r_1,\dots,r_k$ recursively by $r_i=\min\{r-r_1-\dots-r_{i-1},\,b_i-a_i\}$.
Then each $r_i$ is a dyadic rational, $0\le r_i\le b_i-a_i$ and $\sum_ir_i=r$. For each $i$ there is a clopen set $W_i\subseteq J_{b_i}\setminus J_{a_i}$ with $\nu(W_i)=r_i$: this is trivial if $r_i=0$ or $r_i=b_i-a_i$, and
otherwise it follows from \eqref{eq:adm} and \eqref{eq:Sigma}. Then $W=\bigcup_iW_i\subseteq V$ and $\nu(W)=r$. Thus $\nu$ is good, and by Akin's theorem quoted above, $\nu=h_*\lambda$ for some $h\in G$.
\end{proof}

The next lemma is a special case of \cite[Theorem 3.2]{Ak1}, so we skip the proof.

\begin{lemma}\label{lem:exist}
For every countable dense set $S\subseteq(0,1)$ there is a full nonatomic measure $\nu$ on $E$ with $\St(\nu)=S$.
\end{lemma}

\begin{proof}[Proof of Theorem \ref{thm:cosets}]
By Corollary \ref{cor:criterion}\ref{critD}, the map $PgL\mapsto\St(g_*\lambda)$ is well defined and injective, and by Lemma \ref{lem:adm} its values are admissible sets. If $S$ is admissible, then
by Lemma \ref{lem:exist} there is a full nonatomic measure $\nu$ with $\St(\nu)=S$, and by Lemma \ref{lem:adm}, $\nu=h_*\lambda$ for some $h\in G$. Hence the map is onto the family of admissible sets. The coset $PL$ corresponds to $\St(\lambda)=\D$.

Let $R\subseteq\D$ have no accumulation point in $(0,1)$, and let $S=\D\setminus R$. Then $S$ is dense in $(0,1)$.
Let $a<b$ in $S\cup\{0,1\}$ and let $r\in\D$, $r<b-a$. Fix real numbers $0<\delta_1<\delta_2<\min\{r,b-a-r\}$. For $\delta\in[\delta_1,\delta_2]$ the points $a+r-\delta$ and $b-\delta$ lie in the compact sets
$[a+r-\delta_2,a+r-\delta_1]\subseteq(a,a+r)$ and $[b-\delta_2,b-\delta_1]\subseteq(a+r,b)$, respectively. These sets contain only finitely many points of $R$, so there is a dyadic rational $\delta\in[\delta_1,\delta_2]$
such that neither $a+r-\delta$ nor $b-\delta$ belongs to $R$. These two numbers are dyadic rationals in $(0,1)$, hence they belong to $S$, and
\[
a<a+r-\delta<b-\delta<b,\qquad (a+r-\delta-a)+(b-(b-\delta))=r .
\]
Thus $r\in\Sigma_S(a,b)$ and $S$ is admissible.

There are $2^{\aleph_0}$ sets $R\subseteq\{2^{-k}:k\ge1\}$, and they yield pairwise different admissible sets $\D\setminus R$. Since $|G|=2^{\aleph_0}$, the number of double cosets is exactly $2^{\aleph_0}$.
Finally, every double coset different from $PL$ is disjoint from the comeager set $PL$ (Theorem \ref{thm:main}\ref{main3}), hence it is meager.
\end{proof}

\section{Topological generators}\label{sec:gen}

In this section we prove Theorem \ref{thm:gen}.

\subsection{Thompson's groups}
Let $V$ be the set of all $g\in G$ for which there exist $n\ge1$ and words $a_1,\dots,a_n$, $b_1,\dots,b_n$ such that both $\{[a_1],\dots,[a_n]\}$ and $\{[b_1],\dots,[b_n]\}$ are partitions of $E$ and
\[
g(a_iy)=b_iy\qquad(i\le n,\ y\in\cs).
\]
The group $V\le G$ is isomorphic to the original Thompson group $V$ (see \cite[Section 2]{BQ}).
We say that such a $g$ is given by \emph{prefix replacement}: on each of the sets $[a_i]$ it replaces the prefix $a_i$ by $b_i$ and leaves the rest of the sequence unchanged. In the definition we may always pass to finer partitions (replace $a_i$ by $a_i0,a_i1$ and $b_i$ by $b_i0,b_i1$).
Since $g$ is determined by the finitely many words $a_i$ and $b_i$, the group $V$ is countable.
Traditionally, Thompson's groups are defined by means of the dyadic rationals: $V$ is introduced as the group of all right-continuous bijections of $[0,1)$ which map dyadic rationals to dyadic rationals and which are piecewise linear with finitely many pieces,
dyadic breakpoints and slopes which are powers of $2$, see \cite[Section 6]{CFP}. Such a bijection is discontinuous at finitely many dyadic rationals; it becomes a homeomorphism when every dyadic rational of $(0,1)$ is doubled, and this leads to the description by prefix replacement given above, which is much cleaner.
The elements of $V$ for which the words can be chosen so that $[a_1]<\dots<[a_n]$ and $[b_1]<\dots<[b_n]$ form Thompson's group $F=V\cap P$.
%here $A<B$ means that every element of $A$ is smaller than every element of $B$.
%Under the identification of $E$ with the unit interval in which every dyadic rational of $(0,1)$ is doubled, $F$ is the group of all piecewise linear homeomorphisms of $[0,1]$ with dyadic breakpoints and slopes which are powers of $2$.
An element of $V$ preserves $\lambda$ if and only if $|a_i|=|b_i|$ for all $i$. In particular, if $\sigma$ is a permutation of the set of all words of length $n$, then
\[
uy\mapsto\sigma(u)y\qquad(|u|=n,\ y\in\cs)
\]
is an element of $V\cap L$, which we denote by $\sigma$ again. We write $(u_1\ u_2\ \dots\ u_k)$ for the cyclic permutation $u_1\mapsto u_2\mapsto\dots\mapsto u_k\mapsto u_1$ which fixes all the other words of the same length.

\begin{lemma}\label{lem:Vdense}
The group $V$ is dense in $G$.
\end{lemma}

\begin{proof}
Let $g\in G$ and $n\ge1$. For every word $w$ of length $n$ the clopen set $g^{-1}[w]$ is the union of finitely many pairwise disjoint sets of the form $[a]$, and the number of these sets can be increased by one (replace $[a]$ by $[a0]$ and $[a1]$); the same is true for $[w]$.
Hence we can write $g^{-1}[w]=[a_{w,1}]\cup\dots\cup[a_{w,k_w}]$ and $[w]=[b_{w,1}]\cup\dots\cup[b_{w,k_w}]$, both unions being disjoint. The map $v$ given by $v(a_{w,i}y)=b_{w,i}y$ for $|w|=n$, $i\le k_w$ and $y\in\cs$ belongs to $V$ and
satisfies $v(g^{-1}[w])=[w]$ for every $w$ of length $n$, that is, $v\in N_n(g)$. By Lemma \ref{lem:top}\ref{topB}, $V$ is dense in $G$.
\end{proof}

\subsection{Proof of Theorem \ref{thm:gen}}
Recall that
\[
p(0y)=00y,\qquad p(10y)=01y,\qquad p(11y)=1y\qquad(y\in\cs)
\]
and $l=(00\ \,01\ \,11\ \,10)$. Clearly $p\in F\subseteq P$ and $l\in V\cap L$ has order $4$. Let $\Gamma$ be the subgroup of $G$ generated by $p$ and $l$; thus $\Gamma\subseteq V$. By Lemma \ref{lem:Vdense} it suffices to show that $\Gamma=V$. We do it in three steps.

\smallskip\noindent\emph{Step 1: the element $a=pl^2p$.} We have $l^2=(00\ \,11)(01\ \,10)$. Hence, for $y\in\cs$,
\begin{align*}
0y&\ \overset{p}{\longmapsto}\ 00y\ \overset{l^2}{\longmapsto}\ 11y\ \overset{p}{\longmapsto}\ 1y, &
10y&\ \overset{p}{\longmapsto}\ 01y\ \overset{l^2}{\longmapsto}\ 10y\ \overset{p}{\longmapsto}\ 01y,\\
110y&\ \overset{p}{\longmapsto}\ 10y\ \overset{l^2}{\longmapsto}\ 01y\ \overset{p}{\longmapsto}\ 001y, &
111y&\ \overset{p}{\longmapsto}\ 11y\ \overset{l^2}{\longmapsto}\ 00y\ \overset{p}{\longmapsto}\ 000y .
\end{align*}
Thus $a\in\Gamma$ is given by
\begin{equation}\label{eq:a}
a(0y)=1y,\qquad a(10y)=01y,\qquad a(110y)=001y,\qquad a(111y)=000y ,
\end{equation}
and $a\in L$.

\smallskip\noindent\emph{Step 2: $\Gamma$ contains every permutation of the words of length $2$.} 
For a word $u$ of length $2$ let $s_u=(u0\ \,u1)$ be the transposition of the two halves of $[u]$. The four involutions $s_u$ commute, and
\begin{equation}\label{eq:conj}
\sigma s_u\sigma^{-1}=s_{\sigma(u)}\qquad\text{for every permutation $\sigma$ of the words of length $2$.}
\end{equation}
By \eqref{eq:a}, $a=\alpha s_{11}$ with $\alpha=(00\ \,10\ \,01\ \,11)$, and $la=\beta s_{11}$ with $\beta=l\alpha=(01\ \,10\ \,11)$.
Applying \eqref{eq:conj} in the form $\sigma s_u=s_{\sigma(u)}\sigma$ repeatedly, we get $(\sigma s_u)^k=s_{\sigma(u)}s_{\sigma^2(u)}\cdots s_{\sigma^k(u)}\,\sigma^k$ for $k\ge1$. Since $\alpha^4=\beta^3=\id$, this gives
\[
a^4=s_{00}s_{01}s_{10}s_{11},\qquad (la)^3=s_{01}s_{10}s_{11},\qquad\text{and so}\qquad s_{00}=(la)^3a^4\in\Gamma .
\]
By \eqref{eq:conj} with $\sigma=l,l^2,l^3$ we obtain $s_{01},s_{11},s_{10}\in\Gamma$, hence also $\alpha=as_{11}\in\Gamma$.
Finally, the permutations $l$ and $\alpha$ of the words of length $2$ generate the full symmetric group on these four words: indeed, the group they generate contains the element $l$ of order $4$ and the element $l\alpha=\beta$ of order $3$, so its order is divisible by $12$, and it contains the odd permutation $l$, so it is not the alternating group.

\smallskip\noindent\emph{Step 3: $\Gamma=V$.} By \cite[Theorem 1.2]{BQ}, the group $V$ is generated by the three elements
\[
(00\ \,01),\qquad(01\ \,10\ \,11)\qquad\text{and}\qquad(1\ \,00),
\]
where $(1\ \,00)$ denotes the transposition of the sets $[1]$ and $[00]$ by prefix replacement, that is, $1y\mapsto00y$, $00y\mapsto1y$ and $01y\mapsto01y$ for $y\in\cs$.
The first two elements are permutations of the words of length $2$, so they belong to $\Gamma$ by Step 2, and so does $\varrho=(00\ \,10)(01\ \,11)$; note that $\varrho(0y)=1y$ and $\varrho(1y)=0y$.
Since $p$ maps $[0]$ onto $[00]$ and $[10]$ onto $[01]$ by prefix replacement, $\pi_0=p^{-1}(00\ \,01)\,p$ is the transposition of the sets $[0]$ and $[10]$ by prefix replacement. As $\varrho$ maps $[0]$ onto $[1]$ and $[10]$ onto $[00]$, we get
\[
(1\ \,00)=\varrho\,\pi_0\,\varrho=\varrho\,p^{-1}(00\ \,01)\,p\,\varrho\in\Gamma .
\]
Hence $\Gamma=V$. \qed

\section{Further results and questions}\label{sec:remarks}

\subsection{Generic homeomorphisms}
Each of the groups $G$, $P$ and $L$ has a comeager conjugacy class, see \cite[Theorem 1.5]{KR}, \cite{Tr} and \cite[Theorem 1.12]{KR}, respectively. We denote these classes by $\mathcal C_G$, $\mathcal C_P$ and $\mathcal C_L$. Recall that every $g\in PL$ has a unique decomposition $g=p_gl_g$ with $p_g\in P$ and $l_g\in L$.

\begin{proposition}\label{prop:generic}
\begin{enumerate}[label=\textup{(\alph*)}]
\item\label{genA} The set of all $g\in PL\cap\mathcal C_G$ such that $p_g\in\mathcal C_P$ and $l_g\in\mathcal C_L$ is comeager in $G$.
\item\label{genB} There are $p\in\mathcal C_P$ and $l\in\mathcal C_L$ such that $pl\notin\mathcal C_G$.
\end{enumerate}
\end{proposition}

Thus the generic homeomorphism of the Cantor set is the product of a generic order preserving homeomorphism and a generic measure preserving homeomorphism, but the two factors have to be in a general position.

\begin{proof}
\ref{genA} Let $\varphi\colon P\times L\to PL$ be the multiplication map, which is a homeomorphism by Theorem \ref{thm:main}\ref{main3}. Since $PL$ is a dense $G_\delta$ subset of $G$, a subset of $PL$ is comeager in $PL$ if and only if it is comeager in $G$. Hence $\varphi^{-1}(\mathcal C_G\cap PL)$ is comeager in $P\times L$,
and so is $\mathcal C_P\times\mathcal C_L$. The image under $\varphi$ of the intersection of these two sets is then comeager as well.

\ref{genB} We need two auxiliary facts.
First, as follows by the description \cite{KupkaOprocha}, no element of $\mathcal C_G$ has a fixed point. Second, the group $L$ acts transitively on the set of pairs of distinct points of $E$.

Now let $p\in\mathcal C_P$ and $l_1\in\mathcal C_L$. Since $p\neq\id\neq l_1$, there are $x\in E$ with $x\neq p(x)$ and $z\in E$ with $l_1(z)\neq z$. Denote $y=p^{-1}(x)$ and find $b\in L$ satisfying $b(z)=x$ and $b(l_1(z))=y$. Then put $l=bl_1b^{-1}\in\mathcal C_L$. Then $l(x)=y$, so $pl(x)=x$ and $pl\notin\mathcal C_G$.
\end{proof}

\subsection{Other measures}
Any two linear orders on $E$ which induce the topology of $E$ are isomorphic. 
%(For such an order, the set of all gaps is a countable dense linear order without endpoints, and $E$ is recovered from it as the set of its cuts, with every gap contributing two points.)
Hence the group $P$ is determined uniquely up to conjugacy in $G$. For measures the situation is different: there are uncountably many full nonatomic measures on $E$ no two of which are homeomorphic \cite{Ak1}. It is therefore natural to ask what happens if $\lambda$ is replaced by another measure.
For a measure $\mu$ let $L_\mu=\{g\in G:g_*\mu=\mu\}$.

\begin{remark}\label{rem:anymeasure}
The equality $G=PL$ fails for every measure in place of $\lambda$: if $\mu$ is a full nonatomic measure on $E$, then $P\cap L_\mu=\{\id\}$, $PL_\mu\neq G$ and $L_\mu P\neq G$.

Indeed, Proposition \ref{prop:akin} shows, exactly as in Corollary \ref{cor:criterion}, that $P\cap L_\mu=\{\id\}$ and that $g\in PL_\mu$ if and only if $\St(g_*\mu)=\St(\mu)$. Fix $c\in\St(\mu)$ and let $Q_n=[0^{n-1}1]$ for $n\ge1$, so that $[0^{n-1}]$ is the disjoint union of $Q_n$ and $[0^n]$.
By Lemma \ref{lem:F}, the numbers $\mu(U)$, where $U$ is a nonempty clopen initial segment of the interval $Q_n$ and $U\neq Q_n$, are dense in $(0,\mu(Q_n))$. Hence we can choose inductively such sets $U_n\subseteq Q_n$ so that
$0<c-\sum_{i\le n}\mu(U_i)<\mu([0^n])$ for every $n$; then $\sum_n\mu(U_n)=c$. The sets $I=\{\zinf\}\cup\bigcup_nU_n$ and $I'=\{\zinf\}\cup\bigcup_n(Q_n\setminus U_n)$ are closed sets without isolated points, $I\cup I'=E$ and $I\cap I'=\{\zinf\}$.
The sets $E_{\xs}$ and $E'=\{x\in E:x\ge\xs\}$, where $\xs=0101\ldots$ is the point from Section \ref{sec:example}, have the same properties with $\xs$ in place of $\zinf$. Since any two Cantor sets with distinguished points are homeomorphic by a homeomorphism respecting the distinguished points,
there is $g\in G$ with $g(I)=E_{\xs}$ and $g(I')=E'$. Then $F_{g_*\mu}(\xs)=\mu(I)=c$, and since $\xs$ is not an endpoint of a gap, $c\notin\St(g_*\mu)$ by Lemma \ref{lem:F}\ref{FA}. Thus $g\notin PL_\mu$ and $g^{-1}\notin L_\mu P$.
\end{remark}

\begin{remark}\label{rem:goodmeasures}
On the other hand, the proofs of Corollary \ref{cor:criterion} and of Theorem \ref{thm:main}\ref{main1}, \ref{main3}, \ref{main4} used only two properties of $\lambda$: that $\lambda$ is good in the sense recalled at the beginning of Section \ref{sec:cosets}, and that $\St(\lambda)=S(\lambda)\setminus\{0,1\}$.
Hence these statements remain valid, with the same proofs, if $\lambda$ is replaced by any good measure $\mu$ on $E$ with $\St(\mu)=S(\mu)\setminus\{0,1\}$, the group $L$ by $L_\mu$, and $\D$ by $S(\mu)\setminus\{0,1\}$.
%(In the proof of the density of $U_d$, the number $r=d-\sigma_{i-1}$ belongs to $S(\mu)$ because $\mu$ is good and $\sigma_{i-1}<d$ are elements of $S(\mu)$; the set $C'$ then exists again because $\mu$ is good.)
%It is not hard to see, using Lemma \ref{lem:exist}, Akin's theorem and the fact that $S(\mu)+\Z$ is a subgroup of the reals for every good measure $\mu$ \cite[Proposition 2.4]{Ak2}, that every good measure on a Cantor space is homeomorphic to a measure on $E$ with this property.
\end{remark}

\begin{remark}\label{rem:DKMN}
Proposition \ref{prop:generic}\ref{genA} does not extend to all good measures, since $L_\mu$ needs not to have a comeager conjugacy class.
Doucha, Kwietniak, Malicki and Niemiec \cite[Theorem 5.7]{DKMN} proved that for a good measure $\mu$ with $S(\mu)\subseteq\mathbb Q$
the group $L_\mu$ has a comeager conjugacy class if and only if $S(\mu)+\Z$ is a subring of $\mathbb R$.
\end{remark}

%\subsection{Universal minimal flows}
%The group $P$ is extremely amenable; this follows from Pestov's theorem \cite{Pe}, see \cite[Theorem 1.2]{GW}. The universal minimal flow of $G$ is $\Phi$ \cite{GW}; in \cite[Section 8]{KPT} it is described as the space of all linear orderings of the algebra of clopen subsets of $E$ which are natural on finite subalgebras. Kechris, Soki\'c and Todorcevic \cite{KST} proved that the universal minimal flow of $L$ is the same space with the restricted action, provided that a certain homogeneous version of the dual Ramsey theorem holds; a proof of the latter was announced in \cite{Mi}.

%The action of $L$ on $\Phi$ is minimal. Indeed, the homeomorphisms used in the proof of \cite[Theorem 1.3(i)]{GW} to establish the minimality of the action of $G$ on $\Phi$ are of the form $wy\mapsto\tau(w)y$, where $\tau$ is a permutation of the set of words of a fixed length, and such homeomorphisms belong to $L$. Theorem \ref{thm:main} shows that nevertheless $L$ does not act transitively on the orbit $Gc_0$: under the bijection $\beta$ of Proposition \ref{prop:22}, the set $Lc_0$ corresponds to $\pi(LP)\subsetneq G/P$.

\subsection{A question}
By Theorem \ref{thm:main}, the groups $P$ and $L$ are not complements of each other in $G$, although $L\cap P=\{\id\}$, and by Remark \ref{rem:anymeasure} the same is true if $L$ is replaced by the stabilizer of any full nonatomic measure.

\begin{question}\label{q:complement}
Does $P$ admit a closed complement in $G$, that is, is there a closed subgroup $K\le G$ with $K\cap P=\{\id\}$ and $G=KP$? Does $L$ admit a closed complement in $G$?
\end{question}

\section*{Acknowledgements}
The author used an AI assistant during the preparation of this paper. It was used to search the literature, to draft and edit parts of the text and the figure, to run computer searches for small generating pairs of Thompson's group $V$ which led to the choice of the generators in Theorem \ref{thm:gen}, and to propose and check arguments, in particular in Sections \ref{sec:gen} and \ref{sec:remarks}.
All statements and proofs were verified by the author, who takes full responsibility for the content.

\bibliographystyle{abbrv}
\bibliography{citace}

\end{document}